\documentclass[11pt]{article}
\usepackage{amsmath}
\usepackage{amsthm}
\usepackage{amssymb}

\newtheorem{theo}{Theorem}[section]

\newtheorem{prop}[theo]{Proposition}
\newtheorem{lemma}[theo]{Lemma}

\newcommand{\GG}{{\cal G}}

\begin{document}
\date{}

\title{Connectivity Graph-Codes}

\author{Noga Alon
\thanks{Princeton University,
Princeton, NJ, USA 
and
Tel Aviv University, Tel Aviv,
Israel.
Email: {\tt nalon@math.princeton.edu}.
Research supported in part by
NSF grant DMS-2154082 and by USA-Israel BSF grant 2018267.}
}

\maketitle
\begin{abstract}
The symmetric difference of two graphs $G_1,G_2$ on the same set of
vertices $V$ is the graph on $V$ whose set
of edges are all edges that belong to exactly one of the two graphs
$G_1,G_2$.  For a fixed graph $H$ call a collection $\GG$ of
spanning subgraphs of $H$ a connectivity code for $H$ if the
symmetric difference of any two distinct subgraphs in $\GG$ is a
connected spanning subgraph of $H$. It is easy to see that the maximum
possible cardinality of such a collection is at most 
$2^{k'(H)} \leq 2^{\delta(H)}$,
where $k'(H)$ is the edge-connectivity of $H$ and $\delta(H)$ is its
minimum degree. We show that equality holds for any $d$-regular
(mild) expander, and observe that equality does not hold in several
natural examples including any large cubic graph, the square of a long
cycle and products of a small
clique with a long cycle.
\end{abstract}
%AMS classification: 05D05, 05D40, 94B25
%Keywords: Graph-codes, Error Correcting Codes, Expanders,
%Connectivity
%\setcounter{page}{1}

\section{Introduction}

The {\em symmetric difference} of two graph $G_1=(V,E_1)$ and $G_2=(V,E_2)$
on the same set of vertices $V$ is the graph $(V,E_1\oplus E_2)$
where $E_1 \oplus E_2$ is the
symmetric difference between $E_1$
and $E_2$, that is, the set of all edges that belong to exactly one
of the two graphs. 

An intriguing variant of the well studied 
theory of error correcting codes (see,
e.g., \cite{MS}) is the investigation of collections $\GG$ of graphs on the
set of vertices $V$ in which  the symmetric difference of every
distinct pair satisfies a prescribed property. The systematic study
of this topic was initiated in \cite{AGKMS}, see also \cite{Al2},
\cite {BGMW} for two recent subsequent papers. If all graphs in the
collection $\GG$ are subgraphs of a fixed graph $H$, and the property
considered is connectivity, we call $\GG$ a connectivity code for 
$H$. Let $m(H)$ denote the maximum possible cardinality of a
connectivity code for $H$. It is clear that no two distinct members
of such a code $\GG$ can have exactly the same intersection with the set of
edges of any nontrivial cut of $H$, implying that 
$m(H) \leq 2^{k'(H)}$, where $k'(H)$ is the edge-connectivity of
$H$. In \cite{AGKMS} it is shown that equality holds if $H$
is the complete graph $K_n$, that is, $m(K_n)=2^{n-1}$. In
\cite{BGMW} it is proved that equality holds also for the
$3$ by $3$ torus $C_3 \times C_3$. This is the (Cartesian) product of two
cycles of length $3$ in which two vertices are adjacent iff they
are equal in one coordinate and adjacent in the other. The edge
connectivity (and degree of regularity) here is $4$, and it
is shown in \cite{BGMW} that $m(C_3 \times C_3)=2^4=16$. 

Our main result in this note is that for any $d$-regular graph $H$
satisfying appropriate expansion properties, 
$m(H)=2^d$. 
\begin{theo}
\label{t11}
There exists an absolute constant $c$ so that the following holds.
Let $d$ and $n$ be integers and let 
$H$ be a $d$-regular graph on $n$ vertices. Suppose that 
for every connected induced subgraph of $H$ on a set $W$ of 
$w \leq n/2$ 
vertices, there are at least $c w \log d$ edges connecting $W$ to
its complement. Then $m(H) =2^d$, that is, the maximum cardinality
of a connectivity code for $H$ is $2^d$.
\end{theo}

We also observe that for several natural examples of
$d$-regular graphs $H$ which are $d$-edge-connected, $m(H) \leq
2^{d-1}$. In particular, for all $t>36$,
$m(C_3 \times C_t) \leq 8$. This answers a problem suggested in
\cite{BGMW}.

The proofs are presented in the next sections. Throughout the note,
all logarithms are in base $2$, unless otherwise specified. To
simplify the presentation we omit all floor and ceiling signs,
whenever these are not crucial.

\section{Expanders}

In this section we prove Theorem \ref{t11}. We make no attempt to
optimize the absolute constant $c$ in the statement and the related
constants in the proofs. The code we construct is a linear code.
We  start with the following simple lemma.
\begin{lemma}
\label{l21}
Let $H=(V,E)$ be a graph. Assign each edge $e \in E$ a vector
$v(e) \in F$, where $F$ is a vector space of dimension $r$ over
$Z_2$. Suppose that for every cut $(S,V-S)=\{e \in E:
e \cap S \neq \emptyset, e \cap (V-S) \neq \emptyset\}$ the set
of vectors $\{v(e), e\in (S, V-S)\}$ spans $F$. Then 
$m(H) \geq 2^r$. 
\end{lemma}
\begin{proof}
Choose a basis of $r$ vectors of $F$ and express each vector
$v(e)$ as a linear combination of the elements of this basis. In
this expression $v(e)$ is a vector in $Z_2^r$. For each 
vector $u \in Z_2^r$, let $G_u$ be the subgraph of $H$ consisting
of all edges $e \in E$ for which the inner product of $u$ and
$v(e)$ (over $Z_2$ ) is nonzero (that is, $1$). The symmetric
difference of any two distinct graphs $G_{u}$ and $G_{u'}$ consists
of all edges $e \in E$ for which the inner product of $v(e)$ with
the nonzero vector $u \oplus u'$ is nonzero. Since 
for every cut of $H$ the vectors
$v(e)$ for $e$ in the cut span $F$, this symmetric
difference must have at least one edge in each cut, implying it is
connected. 
\end{proof}

\noindent
{\bf Remark:}\, It is not difficult to see that the condition in 
the last lemma is equivalent to the existence of a {\em linear} 
connectivity code of size $2^r$ for $H$.  
This equivalence is not needed for our purpose here.

In order to prove Theorem \ref{t11} using the above lemma our
objective is to show that for any graph $H$ as in the theorem it is
possible to assign each edge $e$ a vector $v(e) \in Z_2^d$ so that
the vectors assigned to the edges of each cut of $H$ span $Z_2^d$.
In particular, the vectors assigned to all edges incident with any
single vertex must form a basis. The expansion properties of the
graph ensure that the cuts which are not $1$-vertex cuts have
significantly more edges than the $1$-vertex cuts, and hence it
seems intuitively simpler to ensure their vectors
span the whole space. The rigorous proof
is probabilistic, assigning vectors randomly to most 
(but not to all) of the edges
of $H$. Since, however, the probability that
$d$ random vectors of $Z_2^d$ form  a basis is exactly
$\prod_{i=1}^{d} (1-2^{-i})$, which is bounded away from $1$ 
(indeed smaller than $1/2$),
special care is needed to ensure that the vectors assigned to all
edges in every $1$-vertex cut form a basis. To do so, we do not
assign vectors randomly to all edges, but only to most of them, and
complete the assignment using the following lemma.
\begin{lemma}
\label{l22}
Let $H=(V,E)$ be a $d$-regular graph, let $E'$ be a subset of $E$
and let $v(e), e \in E'$ be an assignment of a vector in 
$Z_2^d$ for any edge $e \in E'$. Suppose that for every vertex $u$
the set of vectors $v(e)$ assigned to all edges in $E'$ that are
incident with $u$ is linearly independent. Then it is possible to
complete the given partial assignment  by assigning a vector
$v(e) \in Z_2^r$ to every edge $e \in E-E'$, so that for every vertex
$u$, the set of vectors $v(e)$ assigned to all $d$ edges incident
with it forms a basis of $Z_2^d$.
\end{lemma}
\begin{proof}
Define the required vectors greedily in an arbitrary order,
maintaining the property that the vectors assigned to all edges
incident with any vertex are linearly independent. When we have to
assign a vector to an edge $uu'$ there are at most 
$2^{d-1}-1$ nonzero vectors that are forbidden in order to ensure
that the vectors assigned to edges incident with $u$ will stay
independent. Similarly, $u'$ forbids at most $2^{d-1}-1$
nonzero vectors. Since $2 (2^{d-1}-1)<2^d-1$
there is always a way to choose 
a vector that can be assigned to $uu'$ maintaining the required
property. This completes the proof.
\end{proof}
We also need the following immediate consequence of Petersen's
Theorem.
\begin{lemma}
\label{l23}
For every even $k$ and every $d \geq k$, any $d$-regular graph 
contains a spanning
subgraph in which every degree is either $k$ or $k-1$
\end{lemma}
\begin{proof}
If $d$ is even this follows by a repeated application of 
Petersen's Theorem \cite{Pe} 
that asserts that any $d$-regular (multi)graph contains 
a $2$-factor. If $d$ is odd, add to it a perfect matching
(repeating existing edges if needed), apply the previous case
to the resulting graph, and remove the edges of the added matching 
chosen to the spanning subgraph.
\end{proof}
The main technical lemma we need for the proof of Theorem
\ref{t11} will be established using the (asymmetric) Lov\'asz
Local Lemma, which we state next.
\begin{lemma}[The Lov\'asz Local Lemma, c.f., \cite{AS}, Chapter 5]
\label{l29}
Let $A_i, i \in I$ be a finite collection of events in an arbitrary
probability space. A dependency graph for these events is a graph
$D$
whose set of vertices is the events $A_i$, where each event
is mutually independent of all the events that are its
non-neighbors. Let $D$ be such a dependency graph, and 
let $N(A_i)$ denote the set of all neighbors of $A_i$ in $D$.
Suppose that
for each event $A_i$ there is a real number $x_i \in [0,1)$ so that
for each $i \in I$
$$
\mbox{Prob}(A_i) \leq x_i \prod_{j \in I, A_j \in N(A_i)}(1-x_j)
$$
Then, with positive probability none of the events $A_i$ occurs.
\end{lemma}
We can now state and prove the main lemma,
in which the constants $1000$ and $8$ can be easily improved.
\begin{lemma}
\label{l24}
Let $H=(V,E)$ be a $d$-regular graph on a set $V$ of $n$ vertices,
where $d \geq 1000$.
Suppose that 
for every connected induced subgraph of $H$ on a set $W$ of 
$w \leq n/2$ 
vertices, there are at least $w(8 \log d+2)$ edges connecting $W$ to
its complement. Then there is a spanning subgraph $H'=(V,E')$ of $H$
in which the degree of every vertex is at least 
$d -3 \log d-2$ and an assignment of a vector
$v(e) \in Z_2^d$ for every edge $e \in E'$ so that the following
two properties hold.
\begin{enumerate}
\item
For every vertex $u \in V$, the set of vectors $v(e)$ assigned
to the edges of $H'$ incident with $u$ is linearly independent.
\item
For every integer $w$, $2 \leq w \leq n/2$, and for every
set $W$ of $w$ vertices so that the induced subgraph of $H$
on $W$ is connected (in $H$), the set of 
vectors $v(e)$ for the vectors $e$ in $H'$ that connect
$W$ to its complement span $Z_2^d$.
\end{enumerate} 
\end{lemma}
\begin{proof}
Let $k$ be the smallest
even number which is at least $d -3 \log d-1$. By Lemma 
\ref{l23} there is a spanning subgraph  $H'=(V,E')$ of $H$
in which every degree is either $k$ or $k-1$. Thus every degree
of $H'$ is at least $d-3 \log d-2$ and at most
$d-3 \log d$.

Fixing this subgraph $H'$, choose for every edge $e$ of
$H'$, independently, a vector $v(e) \in Z_2^d$ 
uniformly at random among all $2^d$ vectors of 
$Z_2^d$ (including the $0$ vector, to simplify the computation).
To complete the proof we show, using the 
asymmetric Lov\'asz Local Lemma  (Lemma \ref{l29} above)
that with positive probability this random choice satisfies
the properties stated in the lemma. We proceed with the details.

For each vertex $u \in V$, let $A(u)$ denote the event that the 
set of random vectors $v(e)$  assigned to the edges of $H'$
incident with $u$ are not linearly independent. 

For every integer $w$, $2 \leq w \leq n/2$, and for every
set $W$ of $w$ vertices of $H$ such that the induced subgraph of 
$H$ on $W$ is connected (in $H$), let $B(W)$ denote the event that the 
set of 
vectors $v(e)$ for the edges $e$ in $H'$ that connect
$W$ to its complement does not span $Z_2^d$. If $|W|=w$ call 
$B(W)$ an event of type $w$.

It is not difficult to upper bound the probabilities of these events.
For each vertex $u \in V$, 
\begin{equation}
\label{e21}
\mbox{Prob}(A(u)) \leq 2^{-3 \log d}.
\end{equation} 
Indeed, let the edges of $H'$ incident with $u$ 
be $e_1, e_2, \ldots ,e_s$, where the numbering is arbitrary.
Then $s \leq d-3 \log d$. The probability that the vector
$v(e_i)$ lies in the span of the previous vectors
$v(e_1), \ldots ,v(e_{i-1})$ is at most $2^{i-1}/2^d$. The required
estimate follows by summing over all $i$, using the fact that 
$s \leq d-3 \log d$.

Next we show that for every event $B(W)$ of type $w$
\begin{equation}
\label{e22}
\mbox{Prob}(B(W)) \leq 2^{-4 w \log d}.
\end{equation} 
It is convenient to split the possible values of $w$ into 
two ranges. If $2 \leq w \leq d/5$ then every vertex of
$W$ has at least $d-w+1 > 0.8 d$ neighbors in $H$ that do not lie in
$W$. 
Among these edges, at most
$3 \log d +2$ edges incident with each vertex of $W$ belong to
$H$ and not to $H'$, implying that there are at least
$w(0.8d-3 \log d-2)$ edges of $H'$ connecting vertices of $W$
to its complement. 
For every nonzero vector $z \in Z_2^d$, the probability that 
all vectors $v(e)$ corresponding to these edges are orthogonal
to $z$ is at most $2^{-w (0.8d-3 \log d-2)} \leq
2^{-(d+4w \log d)}$, where here we used the fact that
for $d \geq 1000$ and $w \geq 2$,
$$
w (0.8d-3 \log d-2) \geq d+4w \log d.
$$
The desired estimate for this case 
follows by the union bound over all $2^d-1<2^d$ choices for the
vector $z$.

If $d/5 \leq w \leq n/2$ then, by
assumption, there are at least $w (8\log d+2)$ edges of $H$ connecting
$W$ and its complement. Among these edges, at most
$3 \log d +2$ edges incident with each vertex of $W$ belong to
$H$ and not to $H'$, implying that there are at least
$5w \log d$ edges of $H'$ 
that connect $W$ and its complement.
For every nonzero vector $z \in Z_2^d$, the probability that 
all vectors $v(e)$ corresponding to these edges are orthogonal
to $z$ is at most $2^{-5 w \log d} \leq 2^{-(d+4w \log d)}$
where the last inequality holds since for $w \geq d/5$
and $d \geq 1000$,
$5w \log d \geq d+4w \log d$.
The desired estimate
follows, as in the previous case,
by the union bound over all $2^d-1<2^d$ choices for the
vector $z$.

In order to apply the local lemma we need to define a dependency
graph $D$ for all events $A(u),B(W)$. To do so we need the known
fact (c.f., e.g., \cite{Al1}) that for every graph with maximum
degree $d$, for every integer $w$, and for every vertex $u$
of the graph, the number of sets of $w$ vertices that contain 
$u$ and induce a connected subgraph is smaller than
$(ed)^w$. Note that each event $A(u)$ is determined by the random
vectors assigned to the edges of $H'$ incident with
it. Similarly, each event $B(W)$ is determined by the random
vectors assigned to the edges of $H'$ connecting $W$ and its
complement. It is clear that the graph on the events in which
two events are connected iff the edge sets whose vectors determine
them intersect is a dependency graph.

It follows that  each event $A(u)$ is independent of
all other events besides at most $d$ other events $A(u')$ and 
besides at
most $d(ed)^w <d^{2w}$ events $B(W)$ of type $w$, for each $2 \leq
w \leq n/2$.
Similarly, each event $B(W)$ of type $w$ 
is independent of all other events besides at most
$wd$ events of the form $A(u)$ and at most 
$wd(ed)^r < w d^{2r}$ events $B(W')$ corresponding to sets 
$W$ of size $r$, for every $2 \leq r \leq n/2$.  This is because
the set of all edges of $H'$ connecting a vertex in $W$ with one in
its complement
cover at most $w$ vertices of $W$ and 
less than $(d-1)w$ vertices of its complement, and each such vertex
lies in at most $(ed)^r$ subsets of size $r$ that induce a
connected subgraph of $H$.

To apply the local lemma define, for each event $A(u)$, a real
$x_u=2^{-2 \log d}=\frac{1}{d^2}$. For each event $B(W)$ 
of type $w$, define
$x_W=2^{-3w \log d}=\frac{1}{d^{3w}}.$ 
To complete the proof using Lemma \ref{l29} it
remains to check the following inequalities.
\begin{enumerate}
\item
For every vertex $u$,
\begin{equation}
\label{e23}
\mbox{Prob}(A(u)) \leq \frac{1}{d^2} (1-\frac{1}{d^2})^d
\prod_{w \geq 2}(1-\frac{1}{d^{3w}})^{d^{2w}}
\end{equation}
\item
For every event $B(W)$ of type $w$
\begin{equation}
\label{e24}
\mbox{Prob}(B(W)) \leq \frac{1}{d^{3w}} (1-\frac{1}{d^2})^{dw}
\prod_{r \geq 2}(1-\frac{1}{d^{3r}})^{wd^{2r}}.
\end{equation}
\end{enumerate}
Inequality (\ref{e23}) follows from (\ref{e21}) and the fact that
$$
(1-\frac{1}{d^2})^d
\prod_{w \geq 2}(1-\frac{1}{d^{3w}})^{d^{2w}}
\geq (1-\frac{1}{d}) \prod_{w \geq 2} (1-\frac{1}{d^w})
\geq 1-\sum_{w \geq 1} \frac{1}{d^w} >1/2>1/d.
$$
Inequality (\ref{e24}) follows from (\ref{e22}) and the fact that
$$
(1-\frac{1}{d^2})^{dw}
\prod_{r \geq 2}(1-\frac{1}{d^{3r}})^{wd^{2r}}
\geq e^{-2w/d} \prod_{r \geq 2} e^{-2w/d^r}=
e^{-2w\sum_{r \geq 1} 1/d^r} \geq e^{-w } > d^{-w}.
$$
This completes the proof of the lemma.
\end{proof}
The proof of the main result, Theorem \ref{t11}, follows quickly
from the assertion of the previous lemmas, as shown next.
\begin{proof}
Let $H=(V,E)$ be a graph satisfying the assumptions of Theorem \ref{t11}
with, say, $c=100$. Note that by the assumption 
$100 \log d \leq d$ implying that $d \geq 1000$.
(It is worth noting also that the proof works, as it is, for 
$c=9$ and the additional assumption that $d \geq 1000$.)
By Lemma \ref{l21} it suffices to show that there is an assignment
of a vector $v(e) \in Z_2^d$ for every edge $e \in E$, so that the 
vectors assigned to the edges of any cut $(S,V-S)$ 
of $H$ span $Z_2^d$. Since any cut contains all edges of a cut in which
at least one side induces a connected subgraph of at most half the
vertices, it suffices to ensure that for every such cut the vectors 
assigned to its edges span $Z_2^d$. By Lemma \ref{l24} there is a
spanning subgraph $H'$ of $H$ and an assignment of a vector in
$Z_2^d$ to each of its edges so that the conclusion of Lemma
\ref{l24} hold. Lemma \ref{l22} ensures that this partial
assignment of vectors can be completed to an assignment of
vectors $v(e) \in Z_2^d$ for every edge $e$ of $H$ so that
the vectors assigned to the edges of any $1$-vertex cut form a
basis. The vectors assigned to the edges of any other cut still span, of
course, $Z_2^d$, as they contain all vectors of edges of $H'$ that 
belong to the cut, which span $Z_2^d$. This completes the proof.
\end{proof}

\section{Graphs admitting only smaller codes}

In this section we observe that there are natural classes of
$d$-regular graphs $H$ with edge connectivity $d$ so that
$m(H)$ is strictly smaller than $2^d$ (and is in fact
at most half of that).
This follows from the following simple lemma.
\begin{lemma}
\label{l31}
Let $H=(V,E)$ be a graph, and let $E_1,E_2, \ldots ,E_s$
be sets of edges of $H$, where each $E_i$ contains at most
$t$ edges. Suppose that for every $1 \leq i<j \leq s$,
the set of edges $E_i \cup E_j$ disconnects $H$, that is,
$H-(E_i \cup E_j)$ is disconnected. 
If $s > (2^t+1)2^{t-1}$ then
$m(H) \leq 2^t$.
\end{lemma}
\begin{proof}
Let $\GG$ be a family of $2^t+1$ spanning subgraphs of $H$. We have to show
that it must contain two distinct members whose symmetric
difference is disconnected. By the pigeonhole principle,
for each fixed $i$, $1 \leq i \leq s$, there is a pair 
$\{G_1^{(i)},G_2^{(i)}\}$  of distinct members of $\GG$ 
that have the same intersection with the set $E_i$.
By a second application of the pigeonhole principle, since
$s> {{|\GG|} \choose 2}$, there are distinct $i,j$ so that
$$\{G_1^{(i)},G_2^{(i)}\}=\{G_1^{(j)},G_2^{(j)}\}.$$
Therefore, the symmetric difference of these two graphs
does not contain any edge of $E_i \cup E_j$ and is thus
disconnected.
\end{proof}
We next describe two families of regular graphs for which the above lemma
implies that the maximum size of a connectivity code is
strictly smaller than the trivial bound that follows from
the edge-connectivity.

The (Cartesian, or box) product $H_1 \times H_2$ of two graphs 
$H_1=(V_1,E_1)$ and $H_2=(V_2,E_2)$ is the graph whose vertex
set is the set $V_1 \times V_2$ where $(v_1,v_2)$ and $(u_1,u_2)$
are connected iff either $v_1=u_1$ and $ v_2,u_2$ are connected in
$H_2$ or $v_2=u_2$ and $v_1,u_1$ are connected in $H_1$.
\begin{prop}
\label{p32}
For every clique $K_t$ and cycle $C_s$, where $s>(2^t+1)2^{t-1}$,
the graph $H_{t,s}$ is $(t+1)$-regular and  its edge connectivity
is $t+1$, but $m(H_{t,s}) \leq 2^t$.
\end{prop}
\begin{proof}
It is clear that $H_{t,s}$ is $(t+1)$-regular. The fact that its
edge connectivity is also $t+1$ is a consequence of the known
fact that the edge connectivity of a product is at least the sum of
the edge connectivities of the two factors (see, e.g.,
\cite{CS}). It also follows from the well known fact that the edge
connectivity of any connected, $d$-regular, vertex transitive graph
is $d$ (see \cite{Ma} or \cite{GR}, pp. 38-39).
The upper bound for $m(H_{t,s})$ follows from 
Lemma \ref{l31} in which the $s$ sets $E_i$ are the sets of edges
connecting the vertices of $H_{t,s}$ in which the second coordinate
is vertex number $i$ of the cycle $C_s$ and the vertices 
in which the second coordinate is vertex number $i+1$ of the cycle
(where addition is modulo $s$).
\end{proof}
Since $K_3=C_3$, a special case of the above result is that for
all $s>36$, the maximum possible cardinality of a connectivity 
code for the torus $C_3 \times C_s$ satisfies
$m(C_3 \times C_s) \leq 8$. This answers a question
raised in \cite{BGMW}, (although the problem of determining
$m(C_t \times C_s)$ for all $t,s$ remains open.)
\begin{prop}
\label{p33}
For $s>4$ let $C_s^{(2)}$ denote the graph obtained from the cycle
$C_s$ by connecting any two vertices of distance at most 
$2$ in the cycle. Then $C_s^{(2)}$ is $4$-regular, its
edge-connectivity is $4$, and if $s>6$ then
$m(C_s^{(2)}) \leq 2^{3}=8$
\end{prop}
\begin{proof}
It is clear that $C_s^{(2)}$ is $4$-regular. The fact that it is
$4$-edge connected (and in fact even $4$ (vertex) connected) is
known, c.f., e.g., \cite{BM} pp 48-49 and also follows from
vertex transitivity. In order to prove an upper
bound for $m(C_s^{(2)})$ we apply Lemma \ref{l31}. For each edge 
$f$ of the cycle $C_s$, let $E_f$  denote the set of all edges
$\{u,v\}$ of the graph $C_s^{(2)}$ for which the unique shortest
path in $C_s$ connecting $u$ and $v$ includes the edge $f$.
Then $ |E_f|=3$ for each of the $s$ edges $f$ of $C_s$. It is
easy to check that for any two distinct $f,f'$, 
$E_f \cup E_{f'}$ disconnects $C_s^{(2)}$. The   required upper
bound for $m(C_s^{(2)})$ thus follows from Lemma \ref{l31}.
\end{proof}
\section{Concluding remarks and open problems}
\begin{itemize}
\item
An $(n,d,\lambda)$-graph $H$ is a $d$-regular graph on $n$ vertices in
which the absolute value of each nontrivial eigenvalue is at most
$\lambda$. It is well known that if $\lambda$ is significantly
smaller than $d$ then any such $H$ has strong expansion properties.
In fact, it suffices to assume that the second largest eigenvalue
of $H$ is at most $\lambda$ (with no assumption about
the most negative eigenvalue). A simple result proved in \cite{AM}
is that for any set $W$ of $w \leq n/2$ vertices of a $d$-regular graph
on $n$ vertices in which the second largest eigenvalue is at most
$\lambda$ there are at least $w(d-\lambda)/2 $ edges connecting
$W$ and its complement. Theorem \ref{t11} thus implies that if
$\lambda \leq d-2c \log d$ then $m(H)=2^d$. Note that this is a
pretty mild assumption on $\lambda$, as it is known that for
every $d$ there are infinitely many bipartite $d$-regular 
graphs in which the second largest eigenvalue is at most
$2 \sqrt{d-1}$, see \cite{MSS}.
\item
Our proof of Theorem  \ref{t11} provides a linear connectivity
code of maximum possible cardinality for any graph $H$ satisfying
the assumptions. It will be interesting to decide if there are
interesting examples of graphs $H$ for which non-linear connectivity codes 
can be larger than linear ones. We note that the  
code of maximum possible cardinality for the complete graph
$K_n$ described in \cite{AGKMS} is linear, and so is the code
of maximum possible cardinality 
for $C_3 \times C_3$ given in \cite{BGMW}.
\item
It may be interesting to study the computational problem of 
computing or estimating $m(H)$ for a given input graph $H$.
As mentioned in the introduction, $m(H)$ is always at most
$2^{k'(H)}$, where $k'(H)$ is the edge connectivity of $H$.
On the other hand, $m(H)$ is always at least
$2^{\lfloor k'(H)/2 \rfloor}$. This is because an immediate 
consequence
of the known result of Nash-Williams about
packing edge-disjoint tress in graphs \cite{NW} implies that $H$ has
at least $k=\lfloor k'(H)/2 \rfloor$ pairwise edge disjoint 
spanning trees $T_i$. The collection of all $2^k$ unions 
$\cup_{i\in I} E(T_i)$ of the edge sets of any subset $I$ of these 
trees is a (linear) connectivity code for $H$. Since $k'(H)$
can be computed in polynomial time, this supplies an 
efficient algorithm for approximating the logarithm of
$m(H)$ up to a factor of (roughly) $2$. 
\item
By the remark in the previous comment, the 
smallest possible value of
$m(H)$ for a graph $H$ with an even edge connectivity $k'=2k$
is at least $2^k$.  It is not too difficult to give examples
showing that this is tight. Indeed, let $s$ be an integer,
$s > 2^{k-1}(2^k+1)$, 
and let $H=H(s,k)$ be a graph obtained
from the vertex disjoint union of $s$ cliques
$K(0), K(1), \ldots K(s-1)$, each of size $2k+1$, by adding a matching
$M_i$ of $k$ edges between $K(i)$ and $K(i+1)$, for all
$0 \leq i \leq s-1$, where $K(s)=K(0)$. 
It is worth noting that by choosing the matchings $M_i$
appropriately we can ensure that the graph is nearly regular, that is,
every degree in it is either
$2k$ or $2k+1$.
It is easy to see
that the edge connectivity of this graph is $2k$. Indeed,
deleting less than $2k$ edges leaves each of the cliques
$K(i)$ connected and leaves at least one edge of every matching
$M_i$ besides at most one, keeping the graph
connected. Thus the edge connectivity is at least
$2k$ and thus $m(H) \geq 2^k$. On the other hand, any
union of two
of the matchings $M_i$ disconnects $H$ and hence
the edge connectivity is exactly $2k$. In addition, by Lemma
\ref{l31}, $m(H) \leq 2^k$ and therefore
$m(H)=2^k.$ 
\iffalse
To see that this is tight we construct
a linear connectivity code for $H$ with $2^k$ codewords as follows.
Let $F$ be the set of all vectors in $Z_2^{k+1}$ with an even
Hamming weight. Thus $F$ is a $k$-dimensional vector space
over $Z_2$. Let $e_i \in Z_2^{k+1}$ be the vector with a single
$1$ in coordinate number $i$. Number the vertices of each
clique $K(j)$ is the graph $H$ by the integers $1,2,\ldots,k$ where
each $i \in [k]=\{1,2,\ldots ,k\}$ appears at least once, and assign
to each edge connecting the vertices numbered $i$ and $j$
of it the vector $e_i+e_j \in F$. In addition
assign the edges of each of the matchings $M_i$ $k$ vectors that
form a basis of $F$. Note that for every nontrivial cut
of one of the cliques, the vectors assigned to its edges span $F$.
Indeed, the edges of the cut form a spanning connected 
subgraph of the clique. For
any two distinct vertices of the clique 
numbered $a$ and $b$, the sum of the
vectors assigned to the edges of any path in the cut from $a$ to
$b$ is $e_a+e_b$. All these vectors span $F$. Therefore, if 
a cut of $H$ splits any $K(i)$ in a nontrivial  way, then the
corresponding vectors span $F$. Otherwise, the cut must fully contain
at least one (in fact at least two) of the matchings $M_i$ and 
hence in this case also
the corresponding vectors span $F$. By Lemma \ref{l21} this shows
that $m(H) \geq 2^k$ (and hence that $m(H)=2^k$), as claimed.
\fi
\item
For $n > d \geq 2$ with $nd$ even, let $f(n,d)$ denote the maximum
possible value of $m(H)$ where $H$ ranges over all $d$-regular graphs
on $n$ vertices. Then $f(n,d) \leq 2^d$ for every $n$, and
the complete graph $K_{d+1}$ shows that $f(d+1,d)=2^d$. 
It is more interesting
to study $f(d)$ defined as the
largest $f$ so that there are infinitely many $d$-regular
graphs $H$ satisfying $m(H) = f(d)$. By Theorem \ref{t11} there exists
an absolute constant $d_0$ so that $f(d)=2^d$ for every $d \geq d_0$.
In the proof given here we have made no attempt to optimize the value of
$d_0$, it is possible that the above holds for all $d \geq 4$
(though our proof here would certainly not give it even if
optimized).
On the other  hand $f(d) < 2^d$ for $d \in \{2,3\}$. Indeed, for
$d=2$ the only connected $2$-regular graph is a cycle. If the number of
its vertices is $3$, namely it is the triangle $K_3$, then 
$m(K_3)=2^2=4$, showing that $f(3,2)=4$. On the other hand,
for all $n>3$, $f(n,2)=2$. The upper bound is the special case
of Proposition \ref{p32} with $t=1$, and the lower bound follows
from the trivial code consisting of the edgeless subgraph and the whole
cycle. For $d=3$, if $H$ is a $3$-regular graph on $n$ vertices, then the
symmetric difference of any two members in  a connectivity code
for $H$ must contain at least $n-1$ edges. As $n-1$ is odd and as $H$
has $3n/2$ edges, the Plotkin  bound (see \cite{Pl} or \cite{MS})
implies that the size of the code is at most
$2 \lfloor \frac{n}{2(n-1)+1-3n/2} \rfloor$. If the number of vertices
$n$ exceeds $6$, this provides an upper bound of $4$ for
$m(H)$. Thus $f(3)\leq 4$. To see that this is an equality and that
in fact $f(2n,3)=4$ for every  odd $n>3$, consider the
following cubic bipartite graph $H_n$.
Its two vertex classes  are $A=\{a_0, a_1, \ldots,
a_{n-1}\}$ and $B=\{b_0, b_1, \ldots ,b_{n-1}\}$. The edges
consist of three matchings $M_0,M_1,M_2$, where $M_j$ consists of all
edges $a_ib_{i+j}$, $0 \leq i \leq n-1$, with $i+j$ computed modulo
$n$. The $4$ members of a connectivity code for $H_n$ are the edgeless
subgraph, and all the three unions of two of the matchings 
$M_i$. This is a linear code, and the symmetric difference of any two
distinct members of it is the union of two of the matchings
$M_i$. It is easy to check that each such union is a Hamilton
cycle and hence connected. 

Recall that a Ramanujan $d$-regular graph is a $d$-regular graph
in which the absolute value of each nontrivial eigenvalue is at most
$2 \sqrt {d-1}$.  By Theorem \ref{t11} it follows that for any such
Ramanujan graph $H$, $m(H) =2^d$ provided $d$ is at least some
$d_0$. It may be interesting to decide if this holds with
$d_0=4$.
\end{itemize}
\vspace{0.2cm}

\noindent
{\bf Acknowledgment:}\,
I thank Jie Ma for helpful discussions.

\end{document}